\documentclass{amsart}

\usepackage{amsmath}
\usepackage{amssymb}
\usepackage{graphicx}

\newcommand{\etale}{$\acute{\textrm{e}}$tale }

\newcommand{\ETALE}{$\acute{\textrm{E}}$TALE }
\newcommand{\rr}{\mathbb{R}^2 }
\newcommand{\ep}{\varepsilon}
\newcommand{\topology}{\mathcal{T}}
\newcommand{\toppunc}{\mathcal{T}_R}
\newcommand{\AT}{\mathcal{A}\mathbb{T}}
\newcommand{\s}{\mathcal{S}^1}

\swapnumbers
\theoremstyle{plain}
\newtheorem{theorem}{Theorem}[section]
\newtheorem{lemma}[theorem]{Lemma}
\newtheorem{proposition}[theorem]{Proposition}

\newtheorem{w-vn}[theorem]{Weyl - von Neumann Theorem}

\theoremstyle{definition}
\newtheorem{definition}[theorem]{Definition}

\theoremstyle{remark}

\newtheorem*{acknowledgements}{Acknowledgements}

\begin{document}

\date{December 4, 2007.}
\title[$C^\ast$-Algebras of Tilings with Infinite Rotational Symmetry]{$C^\ast$-Algebras of Tilings with Infinite Rotational Symmetry}
\author[Michael F. Whittaker]{Michael F. Whittaker}
\address{MICHAEL F. WHITTAKER, Mathematics and Statistics, University of Victoria, Canada}
\email{mfwhittaker@gmail.com}
\curraddr{School of Mathematics, University of Wollongong, Australia}

\begin{abstract} 
A tiling with infinite rotational symmetry, such as the Conway-Radin Pinwheel Tiling, gives rise to a topological dynamical system to which an \etale equivalence relation is associated. A groupoid $C^\ast$-algebra for a tiling is produced and a separating dense set is exhibited in the $C^\ast$-algebra which encodes the structure of the topological dynamical system. In the case of a substitution tiling, natural subsets of this separating dense set are used to define an $\AT$-subalgebra of the $C^\ast$-algebra. Finally our results are applied to the Pinwheel Tiling.
\end{abstract}


\maketitle

\section*{INTRODUCTION}

A tiling of the plane refers to the covering of the $2$-dimensional real vector space with polygons such that the polygons only intersect on their borders. A tiling of the plane gives rise to a compact metric space endowed with a continuous action of $\rr$ which is a topological dynamical system. The connections between topological dynamical systems and operator algebras are well established \cite{GPS}. On one hand, dynamical systems are used to produce interesting $C^\ast$-algebras and on the other hand, invariants from operator algebras can be used to classify dynamical systems. 

Associating a $C^\ast$-algebra with a tiling began with a paper of Bellissard \cite{Bel}. In 1995, Johannes Kellendonk \cite{Kel1} exhibited a groupoid structure on the dynamical system associated with a tiling and defined a $C^\ast$-algebra using Renault's groupoid $C^\ast$-algebras. The beauty of Kellendonk's construction was manifested in his reduction of the complexity, while preserving the structure, of the dynamics on the system which provided for the construction of a geometric separating subset of the $C^\ast$-algebra as well as the formulation of natural $C^\ast$-subalgebras. Jared Anderson and Ian Putnam \cite{AP} later used topological methods to compute invariants of the dynamical system of a tiling, to which they associated a different $C^\ast$-algebra, and described these invariants in the framework of $K$-theory for $C^\ast$-algebras. Finally, we note the article \cite{PK} which gives a somewhat expository treatment of the articles \cite{AP} and \cite{Kel1}, and serves as an excellent starting point for the definition of $C^\ast$-algebras from a tiling. In \cite{PK} the $C^\ast$-algebras produced in \cite{AP} and \cite{Kel1} were shown to be strongly Morita equivalent.

In the present article, Kellendonk's construction of a $C^\ast$-algebra from a tiling \cite{Kel1,PK} is extended to include tilings with infinite rotational symmetry. By now there are many examples of tilings with infinite rotational symmetry \cite{Fre}, but the most studied being the Conway-Radin Pinwheel Tiling \cite{Rad}. The Pinwheel Tiling is built from a substitution on a finite base set of tiles called proto-tiles. Two iterations of the substitution on a proto-tile produces 25 tiles including a copy of the original proto-tile in the same orientation and a copy of the proto-tile which has been rotated at an irrational angle with respect to the original proto-tile. This fact implies that in a Pinwheel Tiling there are tiles which appear in an infinite number of distinct orientations. The reader who is unfamiliar with tiling theory may want to read the first three paragraphs of section \ref{Pinwheel Tiling} in order to have an example at hand for the constructions in the first four sections.

We begin by presenting many of the standard definitions relating to tilings with sight modifications to admit tilings with infinite rotational symmetry. Of course, these modifications appear elsewhere in the literature in various guises \cite{BG,ORS,RS,Rad,Sad}, but our modifications anticipate the construction in the sequel. The dynamical system associated with a tiling is produced and the topological structure is discussed.

The construction of a $C^\ast$-algebra begins by producing a compact metric space based on a tiling which is a closed transversal, in the sense of \cite{MRW}, to the action of $\rr$. Considering translational equivalence leads to an \etale equivalence relation, viewed as a groupoid, and allows us to apply Renault's groupoid $C^\ast$-algebras to produce a $C^\ast$-algebra. A separating dense set is contained in the $C^\ast$-algebra and elements of this set are related by a commutation relation.

The $C^\ast$-algebra of a tiling with infinite rotational symmetry admits a particularly nice subalgebra when the tiling is produced from a substitution. We briefly discuss substitution tilings and exploit the substitution to define an $\AT$ $C^\ast$-subalgebra of our $C^\ast$-algebra. An $\AT$-algebra is the inductive limit of matrix algebras with entries that are continuous functions on the circle. Such algebras have been shown to be classifiable whenever they are simple \cite{Ror2}.

In the final section, we apply our results to the Pinwheel Tiling and explicitly define the maps appearing in the construction of the $\AT$-algebra. We then show that the $\AT$-algebra for the Pinwheel Tiling is simple and comment on the $K$-theory of the inductive limit. We conclude with some open questions.

\section{THE GEOMETRY OF TILINGS}\label{geometry}

Let $\rr$ denote the $2$-dimensional Euclidean vector space and $B(x,r)$ an open ball with radius $r$ centered at $x$ in $\rr$. Given a subset $X$ in $\rr$, the translate of $X$ by a vector $y$ in $\rr$ is given by $X+y=\{x+y |x \in X\}$.

\begin{definition}\label{tiling}
A {\em tiling} $T$ of $\rr$ is a countable collection of closed polygonal subsets $\{t_1,t_2,\cdots\}$ of $\rr$, called tiles, whose union covers $\rr$ and the pairwise intersection of the interiors of any two nonequal tiles is empty.
\end{definition}

A tiling is a collection of subsets of $\rr$, so the translate of a tiling $T$ by a vector $x$ in $\rr$ is defined as $T+x=\{t + x | t \in T\}$. Notice that any translate of a tiling is also a tiling. To create a dynamical system we consider the set of all translates of a tiling $T$ which is denoted $T + \rr = \{T+x | x \in \rr \}$. We aim to complete this set in a suitable metric on tilings.

The construction of a metric on $T+\rr$ requires some preliminary notions. To begin, define a group $\Gamma$ as a closed subgroup of the group of orientation preserving isometries on $\rr$ containing all translations, which we write as a semi-direct product
\[ \rr \subseteq \Gamma \subseteq \rr \rtimes \s \]
where $\s$ is identified with the group $SO(2,\mathbb{R})$. Elements in $\s$ are rotations about the origin of $\rr$ and in the sequel these elements will be denoted by $R_\theta$ where
\[ R_\theta = \left(
\begin{array}{cc}
 \cos\theta & -\sin\theta \\
 \sin\theta &  \cos\theta
\end{array} \right) . \]
For $(x,R_\theta)$ and $(y,R_\phi)$ in $\Gamma$ we have group operations
\[ (x,R_\theta)\cdot (y,R_\theta) = (x + R_\theta (y), R_\theta \cdot R_\phi) \quad \textrm{ and } \quad (x,R_\theta)^{-1} = (R_\theta^{-1}(-x),R_\theta^{-1}) \]
and metric
\[ d((x,R_\theta),(y,R_\phi)) = |x-y| + \left( \sum_{i,j=1}^2 (a_{ij}-b_{ij})^2 \right)^{1/2} \]
where $R_\theta = (a_{ij})$ and $R_\phi = (b_{ij})$ in the matrix form given above. Furthermore, we allow an element $(x,R_\theta)$ in $\Gamma$ to act on a tiling $T$ via $(x,R_\theta)\cdot T = R_\theta(T+x)$.
We shall often abuse notation and use $T+x$ for the tiling $(x,R_0)\cdot T$ and $R_\theta(T)$ for the tiling $(0,R_\theta)\cdot T$.

\begin{definition}\label{patch}
A {\em patch} $P$ in a tiling $T$ is a finite subset of tiles in $T$. The notation $T \cap B(x,R)$ will denote the patch of tiles completely contained in the ball of radius $R$ about the point $x$ in $\rr$.
\end{definition}

\begin{definition}\label{Tiling Metric}
Suppose $\Lambda$ is any collection of tilings and $\Gamma$ is a closed subgroup of the orientation preserving isometries of $\rr$. If $T$ and $T^\prime$ are tilings in $\Lambda$, we define the {\em tiling metric}, $d_\Gamma$, on the $\Lambda$ as follows. For $0<\ep<1$ we say the distance between $T$ and $T^\prime$ is less than $\ep$ if we can find elements $(x,R_\theta)$ and $(x^\prime,R_{\theta^\prime})$ in $\Gamma$ such that $d((x,R_\theta),(0,R_0))<\ep$, $d((x^\prime,R_{\theta^\prime}),(0,R_0))<\ep$, and $R_\theta(T+x) \cap B(0,\ep^{-1}) = R_{\theta^\prime}(T^\prime+x^\prime) \cap B(0,\ep^{-1})$. Now $d_\Gamma(T,T^\prime)$ is the infimum over the set consisting of each such $\ep$ satisfying the above hypothesis. If no such $\ep$ exists we define $d_\Gamma(T,T^\prime)$ to be $1$.
\end{definition}

The distance between two tilings is small when the tilings have the same pattern on a large ball about the origin up to a small orientation preserving isometry in $\Gamma$. Of course, the collection of tilings we are concerned about is $T+\rr$.

\begin{definition}\label{continuous hull}
 Given a tiling $T$, the {\em continuous hull} of $T$, denoted by $\Omega_T$, is the completion of $T+\rr$ in the tiling metric $d_\Gamma$.
\end{definition}

We note that every element in the continuous hull is a tiling of $\rr$ and the metric $d_\Gamma$ extends to $\Omega_T$. Moreover, for every tiling $T$ in $\Omega_T$ and every $x$ in $\rr$, $T+x$ is also an element of the continuous hull. From this it follows that $\rr$ is a continuous group action on $\Omega_T$ and we are interested in the dynamics of the system $(\Omega_T, \rr)$. For further details see \cite{PK}.

The local structure of a single tiling, more accurately the patterns in each tiling, play a large role in the global structure of the continuous hull. We briefly present the notion of finite local complexity of a tiling and its consequences.

\begin{definition}\label{Finite Local Complexity}
Suppose $\Lambda$ is a collection of tilings, we say $\Lambda$ has {\em finite local complexity} with respect to the group $\Gamma$ if, for all $R>0$, the set $\{T \cap B(x,R) | T \in \Lambda \textrm{ and } x \in \rr \}\diagup\Gamma$ is finite.
\end{definition}

\begin{lemma}
The continuous hull of a tiling $\Omega_T$ has finite local complexity with respect to $\Gamma$ if for all $R>0$, the set $\{T^\prime \cap B(0,R) | T^\prime \in \Omega_T \}\diagup \Gamma$ is finite.
\end{lemma}

\begin{theorem}[\cite{RW}] \label{FLC implies compact}
If $T$ is a tiling satisfying finite local complexity with respect to $\Gamma$, then the metric space $( \Omega_T , d_\Gamma )$ is compact. 
\end{theorem}

The choice of the group $\Gamma$ is contingent on the original tiling $T$ in the sense that we would like to choose $\Gamma$ to be the smallest subgroup of the orientation preserving isometries of $\rr$ such that $\Omega_T$ has finite local complexity. For example, the Penrose Tiling has finite local complexity for $\Gamma = \rr$, see \cite{AP}. However, the predominant example of a tiling with infinite rotational symmetry, the Pinwheel Tiling, has finite local complexity only when $\Gamma$ is the full group of orientation preserving isometries of $\rr$. We shall say more about this in Section \ref{Pinwheel Tiling}. Finite local complexity also allows for the description of a base set of tiles from which a tiling is created, called proto-tiles.

\begin{definition}\label{proto-tiles}
Given $\{ p_1 , p_2 , \cdots , p_n \}$, a finite and non-empty collection of closed polygonal subsets of $\rr$, we say that $\{ p_1 , p_2 , \cdots , p_n \}$ is a set of proto-tiles for a collection of tilings $\Lambda$ if for each tile $t$ in $T$ such that $T$ is in $\Lambda$ we have $t = \gamma(p_i)$ for some $\gamma$ in the group $\Gamma$, $i \in \{ 1 , 2 , \cdots , n \}$.
\end{definition}

\begin{definition}\label{aperiodic}
We say that a tiling $T$ is {\em aperiodic} if $T + x \neq T$ for every non-zero $x$ in $\rr$. Furthermore, we say that the continuous hull $\Omega_T$ is {\em strongly aperiodic} if $\Omega_T$ contains no periodic tilings.
\end{definition}

Our interest lies in tilings where the continuous hull is strongly aperiodic. In the sequel, we will construct a groupoid $C^\ast$-algebra from a subset of the continuous hull. Strongly aperiodic tiling systems allow us to consider translational equivalence classes which are viewed as the leaves of a foliation with a closed transversal that is homeomorphic to the product of a Cantor set and a circle in the infinite rotation case. The following theorem gives conditions for when the continuous hull is strongly aperiodic.

\begin{theorem}[\cite{AP}]\label{strongly aperiodic condition}
If the action of $\rr$ on $\Omega_T$ is minimal and $T$ is aperiodic, then $\Omega_T$ is strongly aperiodic.
\end{theorem}

\section{THE \ETALE EQUIVALENCE RELATION OF A TILING}\label{etale equiv section}

In this section we generalize the construction of Kellendonk \cite{Kel1,PK} to include tilings with infinite rotational symmetry such as the Pinwheel tiling. The basic idea is to use Kellendonk's punctured hull to define an \etale equivalence relation which encodes the dynamics of the system $(\Omega_T,\rr)$. Once we have shown that we have an \etale equivalence relation, we construct, using Renault's groupoid $C^\ast$-algebras \cite{Ren}, a $C^\ast$-algebra on the dynamical system associated with a tiling.

Suppose $X$ is a set and $R$ is an equivalence relation on $X \times X$. The equivalence relation $R$ is endowed with a natural groupoid structure as follows. Multiplication is partially defined on $R$ with $(x,y)\cdot(w,z)=(x,z)$ if and only if $y=w$ and inverses are defined by $(x,y)^{-1} = (y,x)$. The range and source map $r:R \rightarrow X$ and $s:R \rightarrow X$ are defined via $r(x,y)=x$ and $s(x,y)=y$. The unit space is the diagonal and is denoted $\Delta = \{(x,x) | x \in X\}$. Furthermore, a triple $(X,R,\topology)$ is said to be an {\em \etale equivalence relation} when $X$ is a compact Hausdorff space and $R$ has been endowed with an \etale topology $\topology$. For our purposes $\topology$ is an {\em \etale topology} on $R$ if the following conditions are satisfied:
\begin{enumerate}
	\item $(R,\mathcal{T})$ is $\sigma$-compact,
	\item $\Delta = \{(x,x) | x \in X \}$ is open in $(R,\mathcal{T})$,
	\item every point $(x,y)$ in $R$ has an open neighborhood $U$ in $(R,\mathcal{T})$ such that $r: U \rightarrow r(U)$ and $s: U \rightarrow s(U)$ are homeomorphisms,
	\item if $U$ and $V$ are open sets in $(R,\mathcal{T})$, then $UV$ is open in $(R,\mathcal{T})$,
	\item if $U$ is an open set in $(R,\mathcal{T})$, then $U^{-1}$ is open in $(R,\mathcal{T})$. 
\end{enumerate}

The running assumptions for the remainder of this paper are:
\begin{itemize}
\item the group $\Gamma$ is equal to the full group of orientation preserving isometries on $\rr$; that is, for any $\gamma$ in $\Gamma$ and any proto-tile $p$ there is a tiling in $\Omega_T$ containing the tile $\gamma(p)$,
\item the continuous hull $\Omega_T$ satisfies finite local complexity with respect to $\Gamma$,
\item the continuous hull $\Omega_T$ is strongly aperiodic.
\end{itemize}
We note that restriction from the case of $\Gamma=\rr \rtimes \s$ to any closed subgroup $\Gamma$ is done in the obvious way.

\begin{definition}[\cite{Kel1}]\label{punctured tiling euclidean}
For a tiling $T$ define a point in the interior of each proto-tile which maximally breaks the symmetry of the proto-tile. Such a point is called a {\em puncture}. For each proto-tile $p_i$ in $\{p_1,p_2, \cdots, p_n\}$, denote the puncture by $x(p_i)$. Since each tile in $T$ is the image of a proto-tile under $\Gamma$, we may extend these punctures to each tile $t \in T$ via $\Gamma$; i.e. if $\Gamma$ relates two tiles $( t_1 = (x,R_\theta)t_2 \, \textrm{for some} \, (x,R_\theta) \in \Gamma )$ then $\Gamma$ relates their punctures in the same way $( x(t_1) = (x,R_\theta)(x(t_2)) )$.
\end{definition}

We remark that a puncture in a proto-tile is said to maximally break the symmetry of the proto-tile if the pointed polygon $(p,x(p))$ has trivial symmetry group.

\begin{definition}[\cite{Kel1}] \label{discrete hull euclidean}
Given the continuous hull of a tiling $\Omega_T$, the {\em discrete hull} of $\Omega_T$, denoted by $\Omega_{punc}$, is defined to be all punctured tilings $T^\prime$ in $\Omega_T$ such that the origin of $\rr$ is a puncture of some tile $t$ in $T^\prime$. Introducing the notation that $T^\prime(0)$ is the tile in $T^\prime$ that contains the origin, $0$ in $\rr$, we have 
\[ \Omega_{punc} = \{ T^\prime \in \Omega_T | x(t)= 0 \textrm{ for } t = T^\prime(0)\} \]
Since punctures are in the interior of tiles, the tile $t$ such that $x(t)=0$ is unique.
\end{definition}

Some notation will be necessary for the remainder of this note. From this point forward we shall consider the proto-tiles $\{p_1, \cdots, p_n\}$ to have their puncture on the origin and to be fixed in a standard orientation. We can also rotate the proto-tiles around the origin and keep track of the rotation. For a tiling $T$ with tile $T(0)$ having puncture on the origin we define $\angle T(0) = \theta$ such that $R_\theta(p_i) = T(0)$ for some $i \in \{1, \cdots,n\}$. A natural subset of $\Omega_{punc}$ is the collection of tilings with a tile sitting on the origin in the same orientation as the proto-tiles; i.e define
\[ \Omega_{punc}^0 = \{T \in \Omega_{punc} | T(0) \in \{p_1, \cdots, p_n\} \}. \]
Now each tiling $T$ in $\Omega_{punc}$ may be written as $R_\theta(T_0)$ for some tiling $T_0$ in $\Omega_{punc}^0$. Furthermore, we define a homeomorphism $\psi:\Omega_{punc} \rightarrow \Omega_{punc}^0 \times \s$ via $\psi(T)=(T_0,R_\theta)$, where $\s$ is the full group of rotations on $\rr$. After stating some properties of the relationship between $\Omega_T$, $\Omega_{punc}$, and $\Omega_{punc}^0$ we shall prove that $\psi$ is a homeomorphism.

\begin{lemma} \label{euclidean on omega punc}
Given the continuous hull of a tiling $\Omega_T$, we have:
\begin{enumerate}
	\item If $T^\prime$ is in $\Omega_T$, then $R_\theta(T^\prime+x)$ is in $\Omega_{punc}^0$ for some $(x,R_\theta)$ in $\Gamma$.
	\item $\Omega_{punc}$ is closed in $\Omega_T$.
	\item There is an $\ep > 0$ such that for any tiling $T$ in $\Omega_{punc}$, $T+x$ is not in $\Omega_{punc}$ for any $0 < |x| < \ep$.
	\item For every nontrivial rotation $R_\theta$ in $\s$ and every $T_0$ in $\Omega_{punc}^0$, $R_\theta(T_0)$ is not in $\Omega_{punc}^0$.
	\item There is an $\ep > 0$ such that for any tiling $T_0$ in $\Omega_{punc}^0$, $R_\theta(T_0+x)$ is not in $\Omega_{punc}^0$ for any $(x,R_\theta)$ in $\Gamma$ such that $0 < d((x,R_\theta) , (0,R_0)) < \ep$.
\end{enumerate}
\end{lemma}

\begin{proof}
Statement (i) is obvious. For (ii) take $T^\prime \in \Omega_T$ a point of closure of $\Omega_{punc}$. There is a sequence $\{T_n\}_{n=1}^\infty \in \Omega_{punc}$ such that $\lim_{n \rightarrow \infty} T_n = T^\prime$. Let $T_n(0) = t_n$ and $T^\prime(0) = t^\prime$ where $x(t_n) = 0, \, \forall \, n \in \mathbb{N}$. Now $\lim_{n \rightarrow \infty} x(t_n) = x(t^\prime)$ which implies $0 = x(t^\prime)$. Whence, $T^\prime \in \Omega_{punc}$. For (iii), if $\mathcal{P}$ is the set of proto-tiles for $\Omega_{punc}$, let $\ep>0$ be smaller than the minimum distance from every proto-tile's puncture to the proto-tile's boundary. Statement (iv) is obvious and (v) follows from (iii).
\end{proof}

\begin{lemma}\label{psi homeo}
The map $\psi: \Omega_{punc} \rightarrow \Omega_{punc}^0 \times \s$ via $\psi(T) = (T_0,R_\theta)$ is a homeomorphism.
\end{lemma}

\begin{proof}
Suppose $\psi(T) = (T_0,R_\theta)=(T_0^\prime,R_{\theta^{\prime}}) = \psi(T^\prime)$. Then $T_0 = T_0^\prime$ and since punctures maximally break the symmetry of tiles, $\s$ acts freely on $\Omega_{punc}^0$ so that $R_\theta=R_{\theta^{\prime}}$ and hence $\psi$ is injective. To show that $\psi$ is surjective, let $(T_0,R_\theta)$ be an element of $\Omega_{punc}^0 \times \s$. Since $\Omega_{punc}$ has finite local complexity and $\Gamma$ is the full group of orientation preserving isometries there is a tiling $T_1$ in $\Omega_{punc}$ such that $T_1 \cap B(0,1) = R_\theta(T_0) \cap B(0,1)$ and we have $d_\Gamma(T_1,R_\theta(T_0)) < 1$. Similarly, for every $n \in \mathbb{N}$, there exists a tiling $T_n \in \Omega_{punc}$ such that $d_\Gamma(T_n,R_\theta(T_0)) < 1/n$. This gives us a Cauchy sequence $\{T_n\}$ in $\Omega_{punc}$ converging to $R_\theta(T_0)$. So $\psi$ is surjective because $\Omega_{punc}$ is complete. Endowing $\Omega_{punc}^0 \times \s$ with the product topology of the tiling metric restricted to the translational case and the metric on $\s$ it follows from (v) in Lemma \ref{euclidean on omega punc} that $\psi$ is isometric. Thus, $\psi$ is a homeomorphism.
\end{proof}

The topology of $\Omega_{punc}$, given by the tiling metric, can now be described using the product structure coming from the homeomorphism $\psi$. Indeed, suppose $P$ is a finite patch in some $T$ in $\Omega_{punc}$ and $t$ is a tile in $P$ then we define the set $U(P,t)$ to be all tilings $T^\prime$ in $\Omega_{punc}$ such that the image of the patch $P - x(t)$ under an element $R_\theta$ of $\s$ is a patch in $T^\prime$ with $R_\theta (P-x(t))$ at the origin. In symbols:
\[ U(P,t) = \{ T^\prime \in \Omega_{punc} | R_\theta(P - x(t)) \subset T^\prime \textrm{ for some } R_\theta \in \s\}. \]
Let $\mathcal{U}$ be the collection of all such sets for a given discrete hull $\Omega_{punc}$.

Let us define the collection $\mathcal{U}_0$ to be the restriction of $\mathcal{U}$ to $\Omega_{punc}^0$. Using the homeomorphism $\psi$ it follows that $\mathcal{U}$ can be identified with $\mathcal{U}_0 \times \s$. It is easily verified that $\mathcal{U}_0$ generates the metric topology of $\Omega_{punc}^0$ which is the tiling metric restricted to $\rr$ in $\Gamma$.

\begin{lemma}\label{cantor set times circle}
The punctured hull of a tiling, $\Omega_{punc}$, is the product of a Cantor set and a circle. Moreover, $\Omega_{punc}^0$ is a Cantor set in the relative topology of $\mathcal{U}_0$.
\end{lemma}

\begin{proof}
We begin by showing that elements of $\mathcal{U}_0$ are closed. Let $U_0(P,t)$ be in $\mathcal{U}_0$ and pick $\ep>0$ so small that $P-x(t) \subset B(0,\ep^{-1})$ and $\ep$ satisfies condition (v) of lemma \ref{euclidean on omega punc}. Suppose $T$ is a point of closure of $U_0(P,t)$ and $\{T_n\}$ is a sequence in $U_0(P,t)$ converging to $T$. For some $N$ we have that $d_\Gamma(T,T_n)<\ep$ for all $n \geq N$. By $(v)$ in lemma \ref{euclidean on omega punc} and finite local complexity it follows that $T\cap B(0,\ep^{-1})=T_n\cap B(0,\ep^{-1})$ for all $n \geq N$. Since $P-x(t) \subset B(0,\ep^{-1})$ it follows that $T$ is in $U_0(P,t)$. So $\Omega_{punc}^0$ is generated by closed and open sets which are invariant under translation via strong aperiodicity. Whence, $\Omega_{punc}^0$ is totally disconnected and has no isolated points. Furthermore, $\Omega_{punc}^0$ is compact and hence a Cantor set.
\end{proof}

The next step is to define translational equivalence on $\Omega_{punc}$. We define the equivalence relation $R_{punc} \subset \Omega_{punc} \times \Omega_{punc}$ as follows:
\[R_{punc} = \{(T,T^\prime) | T \in \Omega_{punc} \textrm{ and } T^\prime = T - x(t) \textrm{ for some } t \in T \}.\]
We want to define a suitable topology on $R_{punc}$ which encodes both the topology on $\Omega_{punc}$ and the translation. The natural topology, which denote $\toppunc$, on $R_{punc}$ is the metric topology defined as follows:
\[d_R((T,T-x(t)),(T^\prime,T^\prime-x(t^\prime))) = d_\Gamma(T,T^\prime) + |x(t) - x(t^\prime)|.\]

In analogy with the notation presented for $\Omega_{punc}$ we may view $R_{punc}$ as $R_{punc}^0 \times \s$ where
\[ R_{punc}^0 = \{(T_0,T_0-x(t)) | T_0 \in \Omega_{punc}^0 \textrm{ and } t \in T_0 \}.\] 

If $P$ is a finite patch in some $T$ in $\Omega_{punc}$ and $t,t^\prime$ are tiles in $P$ then we define the set $V(P,t,t^\prime)$ to be all elements $(T,T^\prime)$ in $R_{punc}$ such that $R_\theta(P-x(t)) \subset T$ for some $R_\theta$ in $\s$ and $T^\prime=T-R_\theta(x(t^\prime)-x(t))$. In symbols:
\[ V(P,t,t^\prime) = \{ (T,T-R_\theta(x(t^\prime)-x(t))) | R_\theta(P-x(t))\subset T \textrm{ for some } R_\theta \in \s \}. \]
Let $\mathcal{V}$ be the collection of all such sets for a given discrete hull $\Omega_{punc}$. Moreover, let $\mathcal{V}_0$ be the restriction of $\mathcal{V}$ to $R_{punc}^0$.

The following lemma is verified directly, and implies that condition (iii) of the \etale topology holds for $R_{punc}$.

\begin{lemma}\label{range and source on V euclidean}
The range and source maps $r,s : R_{punc} \rightarrow \Omega_{punc}$ are local homeomorphisms such that, for $V(P,t,t^\prime)$ in $\mathcal{V}$,
\[ r(V(P,t,t^\prime)) = U(P,t) \textrm{ and } s(V(P,t,t^\prime)) = U(P,t^\prime) \]
where $U(P,t)$ and $U(P,t^\prime)$ are in $\mathcal{U}$.
\end{lemma}

Since the range and source maps are homeomorphisms from $\mathcal{V}$ to $\mathcal{U}$ it follows that elements of $\mathcal{V}$ are compact and open via lemma \ref{cantor set times circle}.

\begin{proposition}\label{R punc is a local action euclidean}
The triple $(\Omega_{punc},R_{punc},\toppunc)$ is an \etale equivalence relation.
\end{proposition}

\begin{proof} We show that $\toppunc$ is an \etale topology on $R_{punc}$. Indeed,
$R_{punc}$ is locally compact from lemma \ref{range and source on V euclidean} since $R_{punc}=\cup\mathcal{V}$. The diagonal of $R_{punc}$ may be written as:
\[ \Delta=\bigcup \{V(P,t,t) | P \textrm{ is a finite patch and } t \in P\}, \]
which is homeomorphic to $\Omega_{punc}$ and hence is compact and open. The fact that the range and source maps are local homeomorphisms is the content of lemma \ref{range and source on V euclidean}. The computation that the product and inverse of open sets in $R_{punc}$ is open is left to the reader, however, we note that both follow easily by noting that open sets in $R_{punc}$ may be written as unions of sets of the form $V_0 \times E$ where $V_0$ is in $\mathcal{V}_0$ and $E$ is an open set in $\s$.
\end{proof}

\section{A $C^\ast$-ALGEBRA ASSOCIATED WITH A TILING}\label{tiling algebra}

In the previous section we defined an equivalence relation with topology coming from the punctured hull and translations in $\rr$. In a sense we have discretized the action of $\rr$ on the continuous hull and we aim to define a groupoid $C^\ast$-algebra based on this structure using Renault's groupoid $C^\ast$-algebras \cite{Ren}. We note here that a construction using a crossed product $C(\Omega_T) \rtimes \rr$ is also possible in this case, see \cite{PK} for details. The key point is that the algebra created here is Morita equivalent to the crossed product and has a much more tractable form. In the final two sections we describe a natural subalgebra in the case of a substitution tiling and apply the construction to the Pinwheel Tiling.

The construction presented here is the standard construction of the reduced $C^\ast$-algebra from an \etale equivalence relation. The $C^\ast$-algebra comes with a generating set of functions consisting of partial isometries which have commutation relations very reminiscent of the irrational rotation algebra.

The reader is reminded of the assumptions presented at the beginning of section \ref{etale equiv section}. Furthermore, we shall denote equivalence classes of $R_{punc}$ by $[T] = \{T^\prime| (T,T^\prime) \in R_{punc}\}$ where $T$ is a representative punctured tiling from the class.

We begin by considering the continuous functions of compact support on $R_{punc}$, denoted $C_c(R_{punc})$, which is a complex linear space. A product and adjoint are defined on continuous compactly supported functions $f$ and $g$ via
\begin{eqnarray}
\nonumber
f\cdot g (T,T^\prime) & = & \sum_{T^{\prime\prime} \in [T]} f(T,T^{\prime\prime})g(T^{\prime\prime},T^\prime) \\
\nonumber
f^\ast(T,T^\prime) & = & \overline{f(T^\prime,T)}
\end{eqnarray}
which endows $C_c(R_{punc})$  with a $\ast$-algebra structure. A compactness argument paired with the fact that the range and source maps are local homeomorphisms implies that the sum in the product above is non-zero on only a finite number of terms. The inductive limit topology on $C_c(R_{punc})$, as a $\ast$-algebra, makes $C_c(R_{punc})$ into a topological $\ast$-algebra which reflects the groupoid structure. The inductive limit topology is described by: a sequence $\{f_n\}$ converges to $f$ in $C_c(R_{punc})$ if and only if there exists a compact subset $K$ in $R_{punc}$ such that the support of $f$ is contained in $K$, the support of $\{f_n\}$ is eventually in $K$, and $\{f_n\}$ converges uniformly to $f$ on $K$.

The construction of the reduced $C^\ast$-algebra of $R_{punc}$ is defined by representing $C_c(R_{punc})$, with the above topological $\ast$-algebra structure, as bounded operators on the Hilbert space $\ell^2([T])$ for some $T$ in $\Omega_{punc}$. In particular, for $T$ in $\Omega_{punc}$, $\xi$ in $\ell^2([T])$, and $T^\prime$ in $[T]$ the induced representation from the unit space  $\pi_{T}:C_c(R_{punc})\rightarrow \mathcal{B}(\ell^2([T]))$ is defined by
\[ (\pi_T(f)\xi)(T^\prime)= \sum_{T^{\prime\prime} \in [T]} f(T^\prime,T^{\prime\prime})\xi(T^{\prime\prime}). \]
It is verified in \cite{Ren} that the induced representations from the unit space are nondegenerate representations. The reduced $C^\ast$-algebra norm is defined, for $f$ in $C_c(R_{punc})$, via
\[ \|f\|_{red} = \sup \{\|\pi(f)\| | \pi \textrm{ is an induced representation}\}. \]
Renault \cite{Ren} goes on to show that this defines a $C^\ast$-norm, called the reduced norm.

\begin{definition}
The completion of $C_c(R_{punc})$ in the reduced $C^\ast$-norm is called the reduced $C^\ast$-algebra of $R_{punc}$ and is denoted $C^\ast(R_{punc})$.
\end{definition}

We have defined only the reduced $C^\ast$-algebra norm because the idea is to get our hands on a $C^\ast$-algebra rather than compare the full and reduced $C^\ast$-algebras. Furthermore, for a large class of tilings, namely tilings arising from a substitution, the \etale equivalence relation is an amenable groupoid making the comparison trivial.

The remaining goal of this section is to define a collection of functions in $C_c(R_{punc})$ whose closed linear span is $C^\ast(R_{punc})$. The collection will consist of partial isometries and allows for easy manipulation of the reduced $C^\ast$-algebra. Moreover, these partial isometries come equipped with a commutation relation similar to the commutation relation in the irrational rotation algebra.

The generating set for our dense subalgebra consists of combinations of partial isometries and a unitary. Viewing $\Omega_{punc}$ as the product of a Cantor set and a circle, the partial isometries determine, up to some $\ep > 0$, the location in the Cantor set with a translation and the unitary determines the location in the circle. Define for each $V(P,t,t^\prime)$ in $\mathcal{V}$ the characteristic function $e(P,t,t^\prime)=\mathcal{X}_{V(P,t,t^\prime)}$ and $z$ to be the function with support on the diagonal of $R_{punc}$ taking a tiling $(T,T)$ to the character of the angle of the tile sitting at the origin of $T$ with respect to the proto-tiles in fixed standard orientation. The explicit descriptions are as follows:
\begin{eqnarray}
\nonumber
e(P,t,t^\prime) (T,T^\prime) & = & \left\{
\begin{array}{ccc}
1&\textrm{if } R_\theta(P-x(t)) \subset T \textrm{ for some } R_\theta \in \s \textrm{ and } \\
 &T^\prime = T - R_\theta(x(t^\prime) - x(t)) \\
0& \textrm{otherwise}
\end{array} \right. \\
\nonumber
z(T,T^\prime) \quad & = & \left\{
\begin{array}{cc}
e^{\angle \, T(0) \, i} & \textrm{if } T = T^\prime \\
0& \textrm{otherwise} 
\end{array} \right. .
\end{eqnarray}
Let  $\mathcal{E}$ be the set of functions 
\[ \mathcal{E} = \{z^k \cdot e(P,t,t^\prime) | P \textrm{ is a patch in } T \textrm{ and } t,t^\prime \in P, k \in \mathbb{Z} \}. \]

\begin{lemma}\label{properties of E}
Let $P$ be a patch in a tiling $T$ in $\Omega_{punc}$ such that $t,t^\prime$ are tiles in $P$, both $e(P,t,t^\prime)$ and $z$ are in $C_c(R_{punc})$. Moreover, functions in the set $\mathcal{E}$ have the following properties:
\begin{enumerate}
	\item $ e(P,t,t^\prime)^\ast = e(P,t^\prime,t)$,
	\item $z \cdot e(P,t,t^\prime) = e^{(\angle \, t - \angle \, t^\prime) \, i} (e(P,t,t^\prime) \cdot z)$,
	\item $e(P_1,t_1,t_1^\prime) \cdot e(P_2,t_2,t_2^\prime) = e(P_1 \cup P_2 , t_1 , t_2^\prime)$ if $t_1^\prime = t_2$ and $P_1, P_2$ agree where they intersect, and
	\item $(z^k \cdot e(P_1,t_1,t_1^\prime)) \cdot (z^l \cdot e(P_2,t_2,t_2^\prime)) = e^{l(\angle \, t_2 - \angle \, t_1)\, i}(z^{k+l} \cdot e(P_1 \cup P_2,t_1,t_2^\prime))$ if $t_1^\prime = t_2$ and $P_1, P_2$ agree where they intersect.
\end{enumerate}
\end{lemma}

\begin{proof}
The facts that $e(P,t,t^\prime)$ and $z$ are in $C_c(R_{punc})$ follow from $V(P,t,t^\prime) \in \mathcal{V}$ being compact and open and the diagonal of $R_{punc}$ being compact and open respectively, see lemma \ref{cantor set times circle} and proposition \ref{R punc is a local action euclidean}. For the properties; part (i) is obvious, we prove parts (ii) and (iii), and part (iv) follows by combining (ii) and (iii).

For (ii) we take the product of the two functions in each order and produce the commutation relation. Indeed, on one hand
\[ z \cdot e(P,t,t^\prime)(T,T^\prime) = \sum_{T^{\prime\prime} \in [T]}z(T,T^{\prime\prime})e(P,t,t^\prime)(T^{\prime\prime},T^\prime) \]
\begin{eqnarray}
	\nonumber
	& = & z(T,T)e(P,t,t^\prime)(T,T^\prime) \\
	\nonumber
	& = & e^{(\angle \, T(0)) \, i} e(P,t,t^\prime)(T,T^\prime) \\
	\nonumber
	& = & \left\{
	\begin{array}{ccc}
	e^{(\angle \, t + \theta)\, i} & \textrm{if } R_\theta(P-x(t))\subset T \textrm{ and } T^\prime = T - R_\theta(x(t^\prime) - x(t)) \\ 
	& \textrm{ for some } R_\theta \in \s \\
	0& \textrm{otherwise} 
	\end{array} \right. \\
	\nonumber
	& = & e^{(\angle \, t + \theta)\, i} e(P,t,t^\prime)(T,T^\prime).
\end{eqnarray}
On the other hand
\[ e(P,t,t^\prime) \cdot z(T,T^\prime) = \sum_{T^{\prime\prime} \in [T]}e(P,t,t^\prime)(T,T^{\prime\prime})z(T^{\prime\prime},T^\prime) \]
\begin{eqnarray}
	\nonumber
	& = & e(P,t,t^\prime)(T,T^\prime)z(T^\prime,T^\prime) \\
	\nonumber
	& = & e^{(\angle \, T^\prime(0))\, i} e(P,t,t^\prime)(T,T^\prime) \\
	\nonumber
	& = & \left\{
	\begin{array}{ccc}
	e^{(\angle \, t^\prime + \theta)\, i} & \textrm{if } R_\theta(P-x(t))\subset T \textrm{ and } T^\prime = T - R_\theta(x(t^\prime) - x(t)) \\
	& \textrm{ for some } R_\theta \in \s \\
	0& \textrm{otherwise} 
	\end{array} \right. \\
	\nonumber
	& = & e^{(\angle \, t^\prime + \theta)\, i} e(P,t,t^\prime)(T,T^\prime).
\end{eqnarray}

For (iii) we have
\[ e(P_1,t_1,t_1^\prime) \cdot e(P_2,t_2,t_2^\prime) (T,T^\prime) = \sum_{T^{\prime\prime} \in [T]} e(P_1,t_1,t_1^\prime)(T,T^{\prime\prime}) e(P_2,t_2,t_2^\prime)(T^{\prime\prime},T^\prime) \]
\[ = \sum_{T^{\prime\prime} \in [T]} \left\{
\begin{array}{cccc}
1&\textrm{if } R_\theta(P_1-x(t_1))\subset T \textrm{ and }R_\phi(P_2-x(t_2))\subset T^{\prime\prime} \\
 &\textrm{ for some } R_\theta , R_\phi \in \s \textrm{ such that }\\ 
 & T^{\prime\prime} = T-R_\theta (x(t_1^\prime)-x(t_1)) \, , \, T^\prime=T^{\prime\prime}-R_\phi (x(t_2^\prime)-x(t_2)) \\
0& \textrm{otherwise}\\
\end{array} \right. \]
\begin{eqnarray}
\nonumber
 &=& \left\{
\begin{array}{ccc}
1&\textrm{if } R_\theta(P_1-x(t_1))\subset T \textrm{ and } R_\phi(P_2-x(t_2))\subset T - R_\theta(x(t_1^\prime) - x(t_1))  \\
 & \textrm{s.t. } T^\prime = T - R_\theta(x(t_1^\prime)-x(t_1))-R_\phi(x(t_2^\prime)-x(t_2)) \textrm{ for some } R_\theta , R_\phi \\
0& \textrm{otherwise}\\
\end{array} \right. \\
\nonumber
&=& \left\{
\begin{array}{ccc}
1&\textrm{if } R_\theta(P_1 \cup P_2-x(t_1))\subset T \textrm{ such that } R_\theta = R_\phi \in \s, \\ 
 & t_1^\prime=t_2, \textrm{ and } T^\prime = T - R_\theta(x(t_2^\prime)-x(t_1)) \\
0& \textrm{otherwise}\\
\end{array} \right.
\end{eqnarray}
$\quad= e(P_1 \cup P_2 , t_1 , t_2^\prime)(T,T^\prime) \textrm{ if } t_1^\prime = t_2$.
\end{proof}

Some remarks are in order. First we note that using the commutation relations in lemma \ref{properties of E} we obtain
\begin{eqnarray}
\nonumber
(z^k \cdot e(P,t,t^\prime)) \cdot (z^k \ast e(P,t,t^\prime))^\ast & = & e(P,t,t) \\
\nonumber
(z^k \cdot e(P,t,t^\prime))^\ast \cdot (z^k \cdot e(P,t,t^\prime)) & = & e(P,t^\prime,t^\prime)
\end{eqnarray}
and both $e(P,t,t)$ and $e(P,t^\prime,t^\prime)$ are projections in $C^\ast(R_{punc})$. Whence, each of the functions in $\mathcal{E}$ is a partial isometry. Suppose $\{p_1, \cdots, p_n\}$ are the proto-tiles of $\Omega_{punc}$, then the function $I = \sum_{i=1}^n e(\{p_i\},p_i,p_i)$ is the identity of the $C^\ast$-algebra $C^\ast(R_{punc})$. Finally, through the induced representation, the function $z^k \cdot e(P,t,t^\prime)$ acts on $\ell^2([T])$ as follows: if $R_\theta(P-x(t)) \subset T^\prime$, $T^\prime$ is in $[T]$, and $\xi$ is in $\ell^2([T])$, then
\[ (\pi_T(z^k \cdot e(P,t,t^\prime))\xi)(T^\prime) = e^{k(\angle \, t + \theta) \, i} \xi(T^\prime - R_\theta(x(t^\prime)-x(t))).  \]

\begin{proposition}
The complex linear span of the functions in $\mathcal{E}$ is a uniformly dense $\ast$-subalgebra of $C^\ast(R_{punc})$.
\end{proposition}

\begin{proof}
To begin we note that elements in the set $\mathcal{E}$ form an algebra under pointwise multiplication. The inductive limit topology ensures that convergent sequences in $C_c(R_{punc})$ are uniformly convergent. It follows from lemma \ref{properties of E} that $\mathcal{E}$ separates points of $R_{punc}$ and does not identically vanish on any point of $R_{punc}$. Whence, by the Stone - Weierstrass Theorem, $\mathcal{E}$ is uniformly dense in $C_c(R_{punc})$ and therefore also dense in $C^\ast(R_{punc})$.
\end{proof}

\section{SUBSTITUTION TILINGS AND $\mathcal{A}\mathbb{T}$-ALGEBRAS}\label{substitutions and AT}

In this section we begin by briefly giving some of the background required for substitution tilings. Given a substitution tiling with the group $\Gamma$ as the full group of orientation preserving isometries whose continuous hull is strongly aperiodic and has finite local complexity, we can form an $\AT$-algebra which is a $C^\ast$-subalgebra of $C^\ast(R_{punc})$. We remind the reader that $\AT$-algebras are inductive limits of matrix algebras whose entries are continuous functions on the circle. In the setting where $\Gamma = \rr$, Kellendonk and Putnam \cite{PK} show that an $AF$-algebra is produced, and the construction presented here extends their construction. An example of a tiling with the above properties is the Pinwheel Tiling, and in the next section we apply the construction presented here to the Pinwheel Tiling.

The reader again is reminded of the assumptions presented at the beginning of section \ref{etale equiv section}. Also recall the definition of the proto-tiles of a tiling, definition \ref{proto-tiles}, and in particular the relationship of the proto-tiles and the group $\Gamma$.

A substitution rule consists of a scaling factor $\lambda>1$ and a substitution $\omega$, such that for each $p_i$ in $\{ p_1 , p_2 , \cdots , p_n \}$, $\omega(p_i)$ is a finite collection of tiles $\{t_1, \cdots , t_l \}$ such that:
\[t_j = \gamma(p_j) \textrm{ for some } \gamma \in \Gamma , p_j \in \{ p_1 , p_2 , \cdots , p_n \}, \]
\[Int(t_j) \cap Int(t_k) = \emptyset , \]
\[\bigcup_{j=1}^l t_j = \lambda p_i. \]

Basically each proto-tile is being divided up into smaller versions of proto-tiles that have been moved by the group $\Gamma$ and then expanded by the scaling factor $\lambda$ so as to be the same size as the original proto-tiles. So for each proto-tile $p_i$ we apply $\omega$ and receive a patch of moved proto-tiles. Now extend the definition of $\omega$ to the image of a proto-tile under $\Gamma$. We write this tile as $t = \gamma(p_i)$ for some $\gamma$ in $\Gamma$, $p_i$ in $\{ p_1 , p_2 , \cdots , p_n \}$ and define $\omega(t)=\gamma(\omega(p_i))$. We may therefore define $\omega(T) = \{ \omega(t) | t \in T\}$ for a tiling $T$.

A substitution rule $\omega$ on a set of proto-tiles $\{p_1,\cdots,p_n \}$ is said to be {\em primitive} if, for some $k \geq 1$, an image under $\Gamma$ of each proto-tile $p_i$ appears inside the patch $\omega^k (p_j)$ for each pair $i,j$ in $\{ 1 , 2 , \cdots , n \}$.

Let us state some facts about substitutions and give references to where these facts are located in the literature. First, Kellendonk and Putnam \cite{PK} show that given a primitive substitution system one can produce a tiling $T$ by iterating the substitution. So the definition of $\Omega_T$ and subsequent definitions are not vacuous. Now if the continuous hull of a tiling has been created from a primitive substitution tiling system, then the substitution is a continuous map from $\Omega_T$ onto $\Omega_T$ \cite{AP}. Finally, Solomyak \cite{Sol} shows that the continuous hull $\Omega_T$ is strongly aperiodic if and only if $\omega$ restricted to $\Omega_T$ is injective.

We now begin the construction of an $\AT$-algebra. To begin, recall that $U(P,t)$ consists of all tilings, up to rotation, with the the tile $t$ at the origin and containing the patch $P$, see section \ref{etale equiv section}. In this section we will be interested in patches arising from the inflation of proto-tiles under the substitution. Indeed, for a proto-tile $p$ we define $U(p,p)$ to consist of all tilings with some rotation of the proto-tile $p$ having puncture on the origin. Notice that the definition of $U(p,p)$ should technically read $U(\{p\},p)$ but we will abuse this notation for the remainder of this note. To extend to inflations of proto-tiles we define, for fixed $N$ in $\mathbb{N}$,
\[ U_N(p,t) = \{\omega^N(T)-R_\theta(x(t))  |  R_\theta(p) \subset T \textrm{ for some } R_\theta \in \s, \, t \in \omega^N(p)\}.\]
At this point we remark that each $U_N(p,t)$ is an element of $\mathcal{U}$ for every $N$ in $\mathbb{N}$, $p$ a proto-tile, and $t$ in $\omega^N(p)$.

Let $\{p_1,p_2,\cdots,p_n\}$ be the set of proto-tiles for a given discrete hull $\Omega_{punc}$ with punctures on the origin. The sets $U(p_1,p_1), U(p_2,p_2), \cdots, U(p_n,p_n)$ are pairwise disjoint and 
\[ U(p_1,p_1) \cup U(p_2,p_2) \cup \cdots \cup U(p_n,p_n) = \Omega_{punc}. \]
Since $\omega$ is injective on $\Omega_{punc}$ whenever $\Omega_{punc}$ is strongly aperiodic, we have, for fixed $N$ in $\mathbb{N}$, sets of the form $U_N(p_i,t)$ are pairwise disjoint for each $p_i$ in $\{p_1,p_2,\cdots,p_n\}$ and $t$ in $\omega^N(p_i)$ and the disjoint union of such sets is $\Omega_{punc}$. This leads us to define, for each $N$ in $\mathbb{N}$, the collection
\[ \mathcal{U}_N = \{U_N(p_i,t) \in \mathcal{U} | p_i \in \{p_1,\cdots,p_n\} \textrm{ and } t \in \omega^N(p_i)\} \]
consisting of pairwise disjoint sets for each $N$ in $\mathbb{N}$ whose union is $\Omega_{punc}$. Following a similar procedure to that of section \ref{etale equiv section} we can extend the collection $\mathcal{U}_N$ to a subcollection $\mathcal{V}_N$ of $\mathcal{V}$. Fix $N$ in $\mathbb{N}$, for each pair of tiles $t$ and $t^\prime$ in the patch $\omega^N(p_i)$ we define $V_N(p_i,t,t^\prime)$ as the pairs of tilings $(\omega^N(T)-R_\theta(x(t)),\omega^N(T)-R_\theta(x(t^\prime)))$ such that $R_\theta(p_i) \subset T$ for some $R_\theta$ in $\s$. Let
\[ \mathcal{V}_N = \{V_N(p_i,t,t^\prime) | p_i \in \{p_1,\cdots,p_n\} \textrm{ and } t,t^\prime \in \omega^N(p_i)\},\]
and notice that $\mathcal{V}_N$ is contained in $\mathcal{V}$ defined in section \ref{etale equiv section} for each $N$ in $\mathbb{N}$.

Let $e_N(p_i,t,t^\prime)$ be the characteristic function of $V_N(p_i,t,t^\prime)$ and let $\mathcal{E}_N$ denote the set of functions
\[ \mathcal{E}_N=\{z^k \cdot e_N(p_i,t,t^\prime) | p_i \in \{p_1,\cdots,p_n\} \textrm{ and } t,t^\prime \in \omega^N(p_i)\}. \]
Observe that $\mathcal{E}_N$ is a proper subset of $\mathcal{E}$ presented in section \ref{tiling algebra}. We may therefore complete the span of the functions $\mathcal{E}_N$ in the reduced $C^\ast$-norm used to construct $C^\ast(R_{punc})$. Let
\[ \mathcal{A}_N = \overline{span}\{z^k \cdot e_N(p_i,t,t^\prime) | p_i \in \{p_1,\cdots,p_n\}, t,t^\prime \in \omega^N(p_i) \textrm{ and } k \in \mathbb{Z}\}. \]

We aim to show that $\mathcal{A}_N$ is a $C^\ast$-algebra which is isomorphic to a matrix algebra with entries that are continuous functions on the circle. We shall require a lemma whose proof is obtained by restricting lemma \ref{properties of E} to $\mathcal{E}_N$.

\begin{lemma}\label{relations}
Fix $N$ in $\mathbb{N}$, for $p_i,p_j \in \{p_1,\cdots,p_n\}$, $t_1,t_1^\prime \in \omega^N(p_i)$, $t_2,t_2^\prime \in \omega^N(p_j)$, and $k_1,k_2 \in \mathbb{N}$ we have the following relations 
\begin{enumerate}
	\item $[z^{k_1} \cdot e_N(p_i,t_1,t_1^\prime)] \cdot [z^{k_2} \cdot e_N(p_j,t_2,t_2^\prime)] = 0$ if $p_i \neq p_j$
	\item $[z^{k_1} \cdot e_N(p_i,t_1,t_1^\prime)] \cdot [z^{k_2} \cdot e_N(p_j,t_2,t_2^\prime)] = 0$ if $p_i = p_j$ and $t_1^\prime \neq t_2$
	\item $[z^{k_1} \cdot e_N(p_i,t_1,t_1^\prime)] \cdot [z^{k_2} \cdot e_N(p_j,t_2,t_2^\prime)] = e^{k_2(\angle \, t_1^\prime - \angle \, t_1)i}z^{k_1+k_2} \cdot e_N(p_i,t_1,t_2^\prime)$ if $p_i = p_j$ and $t_1^\prime = t_2$
\end{enumerate}
\end{lemma}

These relations imply that the closure of the span of $\mathcal{E}_N$ in the reduced norm is a $C^\ast$-algebra for each $N$ in $\mathbb{N}$. Moreover, from these relations one may deduce a great deal about the structure of the $C^\ast$-algebra $\mathcal{A}_N$. Our first observation is that for each $p_i$ in $\{p_1,\cdots,p_n\}$ there are a finite number of tiles in $\omega^N(p_i)$. Thus, for fixed $N$ and $p_i$ the closed span of $\{e_N(p_i,t,t^\prime) | t,t^\prime \in \omega^N(p_i)\}$ is finite dimensional. Moreover, relations (ii) and (iii) imply that if we fix both $N$ and $p_i$ and define
\[ \mathcal{A}_{(N,p_i)} = \overline{span}\{z^k \cdot e_N(p_i,t,t^\prime) | t,t^\prime \in \omega^N(p_i) \textrm{ and } k \in \mathbb{N} \} \]
then we obtain a homomorphism from $\mathcal{A}_{(N,p_i)}$ into the algebra of $m \times m$ matrices with entries in $C(\mathbb{T})$, where $m$ is the number of tiles in $\omega^N(p_i)$ and $\mathbb{T}$ denotes the circle $\{z \in \mathbb{C} | |z|=1\}$. This follows from the fact that the convolution product and adjoint on $\mathcal{E}_N$ mimic the matrix operations on $C(\mathbb{T}) \otimes \mathbb{M}_m$. Furthermore, the first relation implies that $\mathcal{A}_{(N,p_i)}$ and $\mathcal{A}_{(N,p_j)}$ are orthogonal when $i \neq j$. We may therefore define
\begin{eqnarray}
\nonumber
\mathcal{A}_N &=& \overline{span}\{z^k \cdot e_N(p_i,t,t^\prime) | p_i \in \{p_1,\cdots,p_n\}, t,t^\prime \in \omega^N(p_i), \textrm{ and } k \in \mathbb{N} \} \\
\nonumber
&=& \bigoplus_{i=1}^n \mathcal{A}_{(N,p_i)}.
\end{eqnarray}
The number of summands in the direct sum is the number of proto-tiles and is therefore independent of $N$. However, the sizes of the algebras will depend on the inflation of each proto-tile and will increase as $N$ increases. Our observations lead to the following result. The proof is quite standard and appears in \cite{Whi}.

\begin{proposition} \label{matrix isomorphism}
Fix $N$ in $\mathbb{N}$. Given the discrete hull of a tiling, $\Omega_{punc}$, there is a $C^\ast$-algebra isomorphism
\[ \psi_N: \mathcal{A}_N \rightarrow \bigoplus_{i=1}^n (C(\mathbb{T}) \otimes \mathbb{M}_{N,p_i}) \]
defined on the dense set $\mathcal{E}_N$ as follows. Let $x \in [0,1)$ and $E_{t,t^\prime}^{p_i}$ be a standard matrix unit in the row $t$ and column $t^\prime$ of $\mathbb{M}_{N,p_i}$. Then,
\begin{eqnarray}
\nonumber
\psi_N(z^k)(x) & = & \sum_{p_i, t \in \omega^N(p_i)}e^{(\angle \, t + 2\pi kx)i} \otimes E_{t,t}^{p_i}\\
\nonumber
\psi_N(e_N(p_i,t,t^\prime))(x) & = & 1 \otimes E_{t,t^\prime}^{p_i}.
\end{eqnarray}
\end{proposition}

The next step is to show that these matrix algebras form an inductive system, for which the inductive limit is an $\mathcal{A}\mathbb{T}$-algebra.

\begin{proposition}\label{inclusion map}
For each $N$ in $\mathbb{N}$, $\mathcal{A}_N \subseteq \mathcal{A}_{N+1}$ and the inclusion map, denoted by $\phi_N$, is defined on the set $\mathcal{E}_N$ by \\
$\phi_N (z^k \cdot e_N(p_i,t^\prime,t^{\prime\prime}))$
\[=\sum_{(p_j,t,R_\theta) \in I(p_i)} z^k \cdot e_{N+1}(p_j,R_\theta(p_i-\lambda^Nx(t))+x(t^\prime),R_\theta(p_i-\lambda^Nx(t))+x(t^{\prime\prime}))\]
where $z^k \cdot e_N(p_i,t^\prime,t^{\prime\prime})$ is in $\mathcal{E}_N$ and $I(p_i)$ denotes the set of triples $(p_j,t,R_\theta)$, where $p_j$ is a proto-tile and $t$ is a tile in $\omega(p_j)$ such that $R_\theta(p_i + x(t)) \in \omega(p_j)$ for some $R_\theta$ in $\s$.
\end{proposition}

\begin{proof}
We begin by showing that
\begin{eqnarray}
\label{41}
U_0(p_i,p_i) & = & \bigcup_{(p_j,t,R_\theta)\in I(p_i)} U_1(p_j,t)
\end{eqnarray}
where $U_0(p_i,p_i)$ is a clopen set in $\mathcal{U}_0$ and $U_1(p_j,t)$ is a clopen set in $\mathcal{U}_1$. Note that the sets in the union are pairwise disjoint since $\omega$ is injective. On one hand, if $T \in U_0(p_i,p_i)$ then $R_\theta(p_i) \subset T$ for some $R_\theta \in \s$. So $T(0)=R_\theta(p_i)$. Since substitution tilings are created with iterations of $\omega$, $T(0)$ is a tile in $R_\theta(\omega(p_j)-x(t))$ for some $(p_j,t,R_\theta) \in I(p_i)$. Thus, $T \in U_1(p_j,t)$ and left containment follows. On the other hand, if $T \in U_1(p_j,t)$ for some $(p_j,t,R_\theta) \in I(p_i)$ then $R_\theta(\omega(p_j)-x(t))\subset T$ and $T(0)=R_\theta(t-x(t))$. Since $t=R_\theta(p_i+x(t))$ we have $T \in U_0(p_i,p_i)$.

From \ref{41} it follows that
\[ e_0(p_i,p_i,p_i) = \sum_{(p_j,t,R_\theta) \in I(p_i)} e_1(p_j,t,t). \]
Applying $\omega^N$ to the equality of sets above we obtain for $t^\prime$ and $t^{\prime\prime}$ in $\omega^N(p_i)$, \\
$e_N(p_i,t^\prime,t^{\prime\prime}))$
\[=\sum_{(p_j,t,R_\theta) \in I(p_i)} e_{N+1}(p_j,R_\theta(p_i-\lambda^Nx(t))+x(t^\prime),R_\theta(p_i-\lambda^Nx(t))+x(t^{\prime\prime})).\]
Finally, we include the function $z^k$ and obtain, \\
$z^k \cdot e_N(p_i,t^\prime,t^{\prime\prime})$
\[=\sum_{(p_j,t,R_\theta) \in I(p_i)} z^k \cdot e_{N+1}(p_j,R_\theta(p_i-\lambda^Nx(t))+x(t^\prime),R_\theta(p_i-\lambda^Nx(t))+x(t^{\prime\prime})).\]
Therefore, the span of functions in $\mathcal{E}_N$ is contained in the span of functions in $\mathcal{E}_{N+1}$, and the equality of sets above shows the required map. This extends to $\mathcal{A}_N$ by taking the span and closing both sides of the equality in the reduced $C^\ast$-norm.
\end{proof}

Combining propositions \ref{matrix isomorphism} and \ref{inclusion map}, we have proved the following.

\begin{theorem}\label{AT-algebra}
The inductive limit
\[ \mathcal{A}\mathbb{T}_T = \overline{ \bigcup_{N=0}^\infty \mathcal{A}_N } \]
is a $\mathcal{A}\mathbb{T}$-algebra; that is, an `approximately finite dimensional circle algebra'.
\end{theorem}

\section{PINWHEEL TILING}\label{Pinwheel Tiling}

The primary example of a tiling with infinite rotational symmetry is the Pinwheel Tiling. In this section we define the Pinwheel Tiling and consider the $\AT$-algebra contained in $C^\ast(R_{punc})$ created from the Pinwheel Tiling, which we denote $\AT_{pin}$ and $\mathcal{A}_{pin}$ respectively. 


The substitution for the pinwheel tiling begins with a $(1,2,\sqrt{5})$-right triangle and it's mirror image. We will denote these two tiles by $p_0$ and $p_1$. An important observation for the pinwheel tiling is that the smallest angle in these two triangles is equal to $\arctan(1/2)$, which is an irrational multiple of $2\pi$ \cite{Rad}. As usual we consider these two tiles to have puncture on the origin and to be in fixed orientation. The substitution $\omega$ for the Pinwheel Tiling is illustrated, for the proto-tile $p_0$, in figure \ref{fig:pinwheel rotation}. For the proto-tile $p_1$ we horizontally flip figure \ref{fig:pinwheel rotation}. A Pinwheel Tiling is created from the substitution by iterating the substitution, see \cite{PK} for a construction of a tiling from a substitution.

Notice that the tile in the center of the patch $\omega(p_0)$ is the tile $R_\theta(p_0)$ modulo translation where $\theta = \arctan(1/2)$. One may now observe that the tile in the center of the patch $\omega^2(p_0)$ is the tile $R_{2\theta}(p_0)$ modulo translation and the tile in the center of the patch $\omega^n(p_0)$ is the tile $R_{n\theta}(p_0)$ modulo translation (see figure \ref{fig:pinwheel rotation} for the first two iterations). Since $\theta$ is irrational and the irrational rotations are dense in the circle we will have tiles appearing in an infinite number of orientations. Moreover, the continuous hull for the pinwheel tiling has tiles appearing in all orientations of $\s$. This implies that the group $\Gamma$ consists of all orientation preserving isometries of $\rr$. It is well known that the Pinwheel Tiling has finite local complexity and the continuous hull is strongly aperiodic in this case \cite{Rad}, \cite{Whi}.

To construct $\AT_{pin}$ we begin by defining a labelling scheme. Observe that each of the patches $\omega^N(p_0)$ and $\omega^N(p_1)$ contains $5^N$ tiles. For $N$ in $\mathbb{N}$, $\mathcal{A}_N$ is now defined to be the closed span of the functions in $\{z^k \cdot e_N(p_i,t,t^\prime)\}$. With iteration of $\omega$ we may label each tile in the consecutive patches with a sequence of length $N$. For example $\omega^1(p_0)$ has tiles labeled $\{t_1,t_2,t_3,t_4,t_5\}$, $\omega^2(p_0)$ has tiles labeled $\{t_{11},\cdots,t_{15},t_{21},\cdots,t_{25},t_{31},\cdots,t_{35},t_{41},\cdots,t_{45},t_{51},\cdots,t_{55}\}$, and so on. In general, for $N$ in $\mathbb{N}$ a function in $\mathcal{E}_N$ will have the form $z^k \ast e_N(p_i,t_l,t_m)$ where $i=1$ or $2$, $k$ is in $\mathbb{N}$, and $l=l_1\cdots l_N$, $m=m_1\cdots m_N$ where $l_j$ and $m_j$ take values in $\{1,2,3,4,5\}$. We aim to use the labelling scheme to define the inclusion map from $\mathcal{A}_N$ into $\mathcal{A}_{N+1}$ and the isomorphism of $A_N$ onto $\mathbb{M}_{N,p_0}(C(\mathbb{T})) \oplus \mathbb{M}_{N,p_1}(C(\mathbb{T}))$ for functions in the generating set $\mathcal{E}_N$.

\begin{figure}
\begin{center}
\includegraphics[width=.95\textwidth]{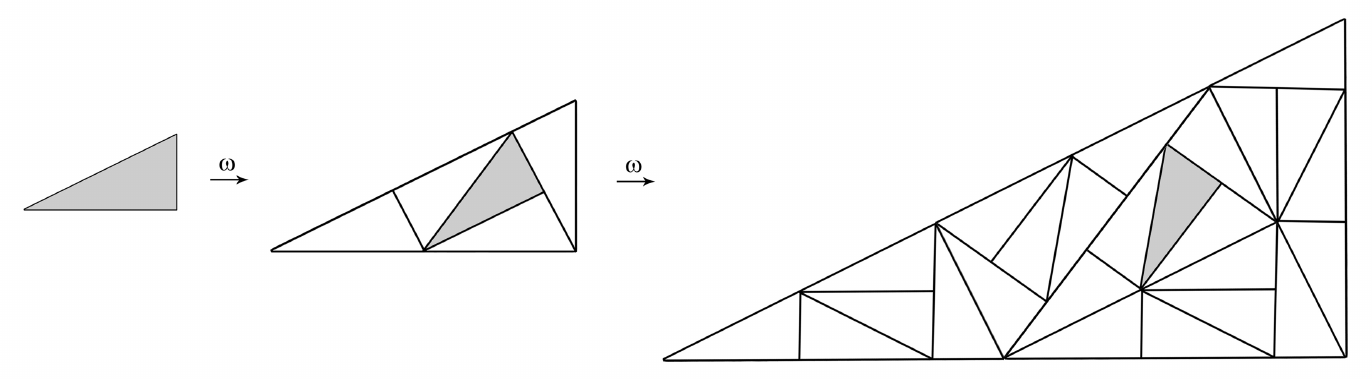}
\caption{The first two substitutions of $p_0$}
\label{fig:pinwheel rotation}
\end{center}
\end{figure}

Let us begin with the inclusion map $\mathcal{A}_N$ into $\mathcal{A}_{N+1}$:
Beginning with the proto-tile $p_i$, let $t_l$ and $t_m$ be tiles in $\omega^N(p_i)$. Then, where $i+1$ is taken $mod2$,
\[ z^k \cdot e_N(p_i,t_l,t_m)) \mapsto \sum_{j=3,4} z^k \cdot e_{N+1}(p_i,t_{j \, l},t_{j \, m}) + \sum_{j=1,2,5} z^k \cdot e_{N+1}(p_{i+1},t_{j \, l},t_{j \, m}). \]

To describe the isomorphism of $A_N$ onto $\mathbb{M}_{N,p_0}(C(\mathbb{T})) \oplus \mathbb{M}_{N,p_1}(C(\mathbb{T}))$ we begin with some notation. For labeled tiles $t_l,t_m$ in $\omega^N(p_i)$ such that $l=l_1 \cdots l_N$ and $m=m_1 \cdots m_N$ we define row $l$ and column $m$ as:
\begin{eqnarray}
\nonumber
\textrm{row } l & = & \textrm{ row } 1+\sum_{j=1}^N 5^{N-j}(l_j-1) \textrm{ in } M_{(N,p_i)}(C(\mathbb{T})) \\
\nonumber
\textrm{column } m & = & \textrm{ column } 1+\sum_{j=1}^N 5^{N-j}(m_j-1) \textrm{ in } M_{(N,p_i)}(C(\mathbb{T}))
\end{eqnarray}
Using this notation, the function $z^k \cdot e_N(p_i,t_l,t_m)$ goes to the direct sum of two matrices with zeros everywhere except in row $l$ and column $m$ of matrix $p_i$ where the entry is $e^{(\angle \, t_{l_1\cdots l_N} + 2k\pi x)i}$ which is a continuous function on the circle and $x \in [0,1)$.

Recall that for an inductive limit $C^\ast$-algebra the map $\phi_{N,M}$ denotes the inclusion of $\mathcal{A}_N$ into $\mathcal{A}_M$ for $M > N$ and the map $\phi_{N,\infty}$ denotes the inclusion of $\mathcal{A}_N$ into the inductive limit. Using a result of Dadarlat, Nagy, Nemethi, and Pasnicu \cite{DNNP} the $C^\ast$-algebra $\AT_{pin}$ is simple if we establish that for every $N$ in $\mathbb{N}$ and for every $f$ in $A_N$ such that $\phi_{N,\infty}(f)$ is non-zero there is an $M>N$ such that $\phi_{N,M}(f)(x) \neq 0$ for every $x$ in $[0,1)$.

\begin{theorem}
$\AT_{pin}$ is simple.
\end{theorem}

\begin{proof}
Take $N$ in $\mathbb{N}$ and $f$ in $\AT_{pin}$ such that $\phi_{\infty,N}(f) \neq 0$. Now there is an open set $U$ about $x$ in $[0,1)$ so that $f$ restricted to some matrix entry in $\mathcal{A}_N$ is non-zero on all of $U$, i.e. for labels $l$ and $m$ and $p_i$ in $\{p_0,p_1\}$, $f_{p_i,l,m} \neq 0$ on all of $U$. Without loss of generality assume $i=0$ so that $f_{p_0,l,m}$ is in row $l$ and column $m$ in the matrix $\mathbb{M}_{(N,p_0)}(C(\mathbb{T}))$. Now applying $\phi_N$ to the function yields
\[ \phi_N(f_{p_0,l,m}) = \left\{
\begin{array}{lllll}
f_{p_1,1l,1m} \neq 0 & \textrm{ on } & R_{-\theta}(U) \\
f_{p_1,2l,2m} \neq 0 & \textrm{ on } & R_{-\theta}(U) \\
f_{p_0,3l,3m} \neq 0 & \textrm{ on } & R_{\theta}(U) \\
f_{p_0,4l,4m} \neq 0 & \textrm{ on } & R_{\theta+\pi}(U) \\
f_{p_1,5l,5m} \neq 0 & \textrm{ on } & R_{-(\theta+\pi/2)}(U)
\end{array} \right\} \in \bigoplus_{j=0}^1 \mathbb{M}_{(N+1,p_j)}(C(\mathbb{T})). \]
Now for $M \geq 2$ even, if we apply $\phi_{N+M,N}$ to the function $f_{p_0,l,m}$ we receive $5^M$ functions in distinct matrix entries in $\mathbb{M}_{(N+M,p_0)}(C(\mathbb{T}))\oplus \mathbb{M}_{(N+M,p_1)}(C(\mathbb{T}))$. In particular, we receive the functions
\[ \left.
\begin{array}{lllll}
f_{p_0,3^Ml,3^Mm} & \neq 0 & \textrm{ on } & R_{M\theta}(U) \\
f_{p_0,2^23^{M-2}l,2^23^{M-2}m} & \neq 0 & \textrm{ on } & R_{(M-2)\theta}(U) \\
\vdots & & & \vdots \\
f_{p_0,2^{M-2}3^2l,2^{M-2}3^2m} & \neq 0 & \textrm{ on } & R_{2\theta}(U) \\
f_{p_0,2^Ml,2^Mm} & \neq 0 & \textrm{ on } & U
\end{array} \right. \in \mathbb{M}_{(N+M,p_0)}(C(\mathbb{T})) \]
and the functions
\[ \left.
\begin{array}{lllll}
f_{p_1,3^{M-2}31l,3^{M-2}31m} & \neq 0 & \textrm{ on } & R_{(M-2)\theta}(U) \\
f_{p_1,2^23^{M-4}31l,2^23^{M-4}31m} & \neq 0 & \textrm{ on } & R_{(M-4)\theta}(U) \\
\vdots & & & \vdots \\
f_{p_1,2^{M-4}3^231l,2^{M-4}3^231m} & \neq 0 & \textrm{ on } & R_{2\theta}(U) \\
f_{p_1,2^{M-2}31l,2^{M-2}31m} & \neq 0 & \textrm{ on } & U
\end{array} \right. \in \mathbb{M}_{(N+M,p_1)}(C(\mathbb{T})) \]
using the notation that row $3^42^2l$ is the labeled row $333322l$. This shows that $\phi_{N,M}(f_{p_0,l,m})$ is nonzero on $U , R_{2\theta}(U) , \cdots , R_{(M-4)\theta}(U) , R_{(M-2)\theta}(U) $ in each of $\mathbb{M}_{(N+M,p_0)}(C(\mathbb{T}))$ and $\mathbb{M}_{(N+M,p_1)}(C(\mathbb{T}))$. Since irrational rotations of the circle are dense, we may find an even $M$ so that $\phi_{N,M}(f_{p_0,l,m})$ is nonzero on all of $[0,1)$.
\end{proof}

Finally, we comment on the $K$-theory of $\AT_{pin}$. Both of the abelian groups $K_0(\AT_{pin})$ and $K_1(\AT_{pin})$ are inductive limits of the system:
\[ \mathbb{Z}\oplus \mathbb{Z} \xrightarrow{\left[
\begin{array}{cc}
 2 & 3 \\
 3 & 2
\end{array} \right]} \mathbb{Z}\oplus \mathbb{Z} \xrightarrow{\left[
\begin{array}{cc}
 2 & 3 \\
 3 & 2
\end{array} \right]} \mathbb{Z}\oplus \mathbb{Z} \xrightarrow{\left[
\begin{array}{cc}
 2 & 3 \\
 3 & 2
\end{array} \right]} \cdots. \]
We note that the connecting matrix is not invertible over the integers. Calculations reveal that $K_0(\AT_{pin}) = K_1(\AT_{pin})$ sit in the following exact sequence which does not split:
\[ 0 \rightarrow \mathbb{Z} \rightarrow K_0(\AT_{pin}) \rightarrow \mathbb{Z}\left[ 1/5 \right] \rightarrow 0,\]
that is, $K_0(\AT_{pin})$ has $\mathbb{Z}$ as an ideal and $\mathbb{Z}\left[ 1/5 \right]$ as a quotient but is not the direct sum of the two groups. We refer the reader to exercise $7.5.2$ in \cite{LM} for a similar computation.

We conclude with two open questions regarding the Pinwheel Tiling. What is the $K$-theory of the $C^\ast$-algebra $\mathcal{A}_{pin}$ and, in particular, can the $K$-theory be computed in terms of cohomology in the sense of \cite{AP}? What is the exact relationship between $\AT_{pin}$ and $\mathcal{A}_{pin}$?

\begin{acknowledgements} 
Robin Deeley, Thierry Giordano, Ian Putnam, and Charles Starling must be thanked for conversations ranging from style to insightful remarks on results. In particular, great acclamation is due to Ian Putnam who supervised my work during my thesis, from which this note is based. This paper was written while visiting the Fields Institute in Mathematical Sciences during the Thematic Program in Operator Algebras.
\end{acknowledgements}


\begin{thebibliography}{99}
  
\bibitem{AP} \textsc{J.E Anderson and I.F. Putnam}, Topological Invariants for Substitution Tilings and their Associated $C^\ast$-algebras, \textit{Ergodic Theory Dynamical Systems}, {\bf 18}(1998), 509--537.

\bibitem{Bel} \textsc{J. Bellissard}, K-theory of $C^\ast$-algebras in solid state physics, \textit{Lecture Notes in Physics}, {\bf 257}(1986), 99--156.

\bibitem{BG} \textsc{R. Benedetti and J.M. Gambaudo}, On the Dynamics of $G$-solenoids: Application to Delone Sets, \textit{Ergod. Th. Dynam. Sys.}, {\bf 23}(2003), 673--691.

\bibitem{Bla} \textsc{B. Blackadar},  \textit{K-Theory for Operator Algebras}, Springer-Verlag, New York 1986.

\bibitem{DNNP} \textsc{M. Dadarlat, G. Nagy, A. Nemethi, and C. Pasnicu}, Reduction of Topological Stable Rank in Inductive Limits of $C^\ast$-Algebras, \textit{Pacific Journal of Math}, {\bf 153}(1992), 323--341.

\bibitem{Dav} \textsc{K.R. Davidson}, \textit{$C^\ast$-algebras by Example}, American Math. Soc., Providence, RI, 1996.

\bibitem{Fre} \textsc{D. Frettl\"{o}h}, \textit{Substitution Tilings with Statistical Circular Symmetry}, Preprint, 2007.

\bibitem{GPS} \textsc{T. Giordano, I.F. Putnam, and C.F. Skau}, Topological orbit equivalence and C*-crossed products, \textit{J. Reine Angew. Math.}, {\bf 469}(1995), 51--111.

\bibitem{Gon} \textsc{D. Gon{\c{c}}alves}, New $C^\ast$-algebras from Substitution Tilings, \textit{J. Operator Theory}, {\bf 57:2}(2007), 391--407. 

\bibitem{Kel1} \textsc{J. Kellendonk}, Noncommutative Geometry of Tilings and Gap Labelling, \textit{Rev. Math. Phys.} {\bf 7}(1995), 1133--1180.

\bibitem{Kel2} \textsc{J. Kellendonk}, The Local Structure of Tilings and their Integer Group of Coinvariants, \textit{Commun. Math. Phys.}, {\bf 187}(1997), 115--157.

\bibitem{PK} \textsc{J. Kellendonk and I.F. Putnam}, Tilings, $C^\ast$-algebras, and K-Theory, \textit{CRM Monograph Series}, {\bf 13} (2000), 177--206.

\bibitem{LM} \textsc{D. Lind and B. Marcus},  \textit{An Introduction to Symbolic Dynamics and Coding}, Cambridge Univ. Press, Cambridge, 1995.

\bibitem{MRW} \textsc{P. S. Muhly, J. N. Renault and D. P. Williams}, Equivalence and isomorphism for
groupoid $C^\ast$-algebras, \textit{J. Operator Th.}, {\bf 17} (1987), 3--22.

\bibitem{ORS} \textsc{N. Ormes, C. Radin, and L. Sadun}, A Homeomorphism Invariant for Substitution Tiling Spaces, \textit{Geom. Dedicata}, {\bf 90}(2002), 153--182.

\bibitem{Pet2} \textsc{G. K. Pedersen}, \textit{$C^\ast$-algebras and their automorphism groups}, Academic Press, London,
1979.

\bibitem{Put} \textsc{I.F. Putnam}, On the $K$-Theory of $C^\ast$-Algebras of Principal Groupoids, \textit{Rocky Mountain Journal of Math.}, {\bf 28}(1998), 1483--1518.

\bibitem{Put1} \textsc{I.F. Putnam}, C*-algebras from Smale spaces, \textit{Canad. J. Math}, {\bf 48}(1996), 175--195.

\bibitem{RW} \textsc{C. Radin and M. Wolff}, Space Tilings and Local Isomorphism, \textit{Geom. Dedicata}, {\bf 42}(1992), 355--360.

\bibitem{RS} \textsc{C. Radin and L. Sadun}, An Algebraic Invariant for Substitution Tiling Systems, \textit{Geom. Dedicata}, {\bf 73}(1998), 21--37.

\bibitem{Rad} \textsc{C. Radin}, The Pinwheel Tilings of the Plane, \textit{Annals of Math.}, {\bf 139}(1994), 661--702.

\bibitem{Ren} \textsc{J.N. Renault}, \textit{A Groupoid Approach to $C^\ast$-algebras}, Lecture Notes in Math., vol. 793, Springer-Verlag, Berlin 1980. 

\bibitem{Ror} \textsc{M. R{\o}rdam, F. Larsen, and N.J. Laustsen}, \textit{An Introduction to K-Theory for $C^\ast$-algebras}, London Math. Soc. Student Texts, vol 49, Cambridge Univ. Press, Cambridge 2000.

\bibitem{Ror2} \textsc{M. R{\o}rdam}, \textit{Classification of Nuclear $C^\ast$-Algebras}, Encyclopaedia of Math. Sci., vol 126 (Operator Algebras VII), Springer-Verlag 2002.

\bibitem{Sad} \textsc{L. Sadun}, Tiling Spaces are Inverse Limits. \textit{Journal of Mathematical Physics} {\bf44}(2003), 5410--5414.

\bibitem{Sol} \textsc{B. Solomyak}, Nonperiodicity implies unique composition for self-similar translationally finite tilings, \textit{Discrete Comput. Geom.}, {\bf 20}(1998), 265--279.

\bibitem{Whi} \textsc{M.F. Whittaker}, \textit{Groupoid $C^\ast$-algebras of the Pinwheel Tiling}, Masters Thesis, University of Victoria, Victoria 2005.

\end{thebibliography}
\end{document}